\documentclass[12pt]{article}
\usepackage[T2A]{fontenc}
\usepackage[utf8]{inputenc}
\usepackage{amssymb,amsmath,amsfonts,amsthm,amscd,latexsym,indentfirst,verbatim,color}
\newtheorem{lemma}{Lemma}
\newtheorem{theorem}{Theorem}
\newtheorem{remark}{Remark}
\newtheorem{example}{Example}
\newtheorem{definition}{Definition}
\newtheorem{corollary}{Corollary}
\newtheorem{proposition}{Proposition}

\newcommand{\Span}{\operatorname{Span}}
\newcommand{\Spec}{\operatorname{Spec}}

\newcommand{\alg}{\operatorname{alg}}
\newcommand{\ad}{\operatorname{ad}}
\def\rank{\mathrm{rank}\,}
\def\tr{\mathrm{tr}}

\begin{document}

\begin{center}
{\Large
Quasi-definite primitive axial algebras of Jordan type half}

\smallskip

Ilya Gorshkov, Vsevolod Gubarev
\end{center}

\begin{abstract}
Axial algebras are commutative nonassociative algebras generated by 
a~finite set of primitive idempotents whose action on an algebra is semisimple, 
and the fusion laws on the products between eigenvectors 
for these idempotents are fulfilled.
We find the sufficient conditions in terms of the Frobenius form 
and the properties of idempotents under which an axial algebra 
of Jordan type half is unital and has a finite capacity.

{\it Keywords}:
axial algebra of Jordan type, Frobenius form, semisimplicity.
\end{abstract}

\section{Introduction}

The notion of an axial algebra appeared in the article of J.~Hall, 
F.~Rehren, and S.~Shpectorov in 2015~\cite{HRS} as a~natural extension 
of Majorana theory earlier developed by A.A.~Ivanov~\cite{Ivanov}. 
Axial algebras generalize associative algebras, Jordan algebras, 
and the Griess algebra whose automorphism group is the Monster group.
The main idea which lies behind the theory of axial algebras 
is to realize finite simple groups as automorphism (sub)groups 
of some finite-dimensional commutative nonassociative algebras.

Primitive axial algebras of Jordan type are the most 
well-studied objects among all axial ones, their fusion laws 
represent the properties of the Peirce decomposition fulfilled 
for all Jordan algebras~\cite{Albert}. 
It is known that every primitive axial algebra~$A$ of Jordan type 
has an invariant normalized bilinear Frobenius form $(\cdot,\cdot)$~\cite{HSS}.
The radical of the form coincides with the maximal ideal of~$A$ 
not containing axes from $A$. An axial algebra is called semisimple 
if the radical of its Frobenius form is zero. 
Despite of the obtained results, the question when an axial algebra 
of Jordan type $1/2$ is finite-dimensional or contains a unit is still open. 
We find the sufficient conditions under which an axial algebra 
of Jordan type $1/2$ containing a unit is semisimple and finite-dimensional. 

A~primitive axial algebra~$A$ of Jordan type $1/2$ is called quasi-definite 
if the equality $(a,b) = 1$ for axes $a,b\in A$ implies $a = b$. 
Simple Jordan algebra $J(f)$ over $\mathbb{R}$ with positive-definite form~$f$ 
and the Jordan algebra $H_n(\mathbb{R})$ of symmetric algebras of order~$n$ 
give examples of quasi-definite primitive axial algebras 
of Jordan type $1/2$ (see~\S3).
A~generalization of quasi-definite primitive axial algebras of Jordan type~$1/2$
is provided by primitive axial algebras 
which have a~quasi-definite linear basis $X$ of axes, i.\,e. 
$(x,y) \neq1$ for all pairwise distinct $x,y\in X$.
We show that the matrix algebra $M_n(F)$ over an infinite field~$F$ 
has a~quasi-definite basis consisting of axes. 
Also, every Matsuo algebra $M_{1/2}(G,D)$ has a~quasi-definite basis of axes (see~\S3).

A primitive axial algebra~$A$ of Jordan type $1/2$ is called strongly axial 
if every semisimple primitive idempotent $e\in A$ such that 
$\Spec(\ad_e)\subset \{0,1/2,1\}$, where $\ad_e$ denotes the operator 
of the multiplication on $e$, is a primitive axis in $A$. 
Thus, in a strongly axial algebra, the fusion laws for 
all semisimple primitive idempotents (not necessarily axes) have to be fulfilled.
Actually, every axial algebra~$A$ of Jordan type $1/2$ that 
is also a~Jordan algebra is strongly axial~\cite[Theorem~6]{Albert}.

We concentrate on the study of the properties of finitely generated 
quasi-definite strongly axial algebras of Jordan type $1/2$.
Let $A$ be such an algebra.
We state that given an axis~$a\in A$, a~subalgebra $A_0(a)$ 
is an axial algebra of Jordan type~$1/2$ 
which has a quasi-definite basis of axes. 

The main results of the work are the following.
Suppose that $A$ is a~finitely generated quasi-definite 
strongly axial algebra of Jordan type $1/2$.
If $A$ is additionally finite-dimensional and semisimple, 
then $A$ is unital (Theorem~\ref{Thm:FD2Unit}).
We say that a unital axial algebra~$A$ has a finite capacity~$k$, 
if its unit~$e$ is respresented as a sum of $k$ pairwise orthogonal axes
and such $k$ is minimal.
We prove that finitely generated quasi-definite strongly axial algebras of Jordan type $1/2$
has a finite capacity (Theorem~\ref{lengthe}).

Let us give a brief outline of the work.
In~\S2, the required preliminaries are stated.
In~\S3, definite and quasi-definite primitive axial algebras 
of Jordan type~$1/2$ are defined and different examples of them are given.

In~\S4, we state that the element
$$
x_a(b) = \frac{2ab - (a,b)a - b}{(a,b)-1}
$$
constructed by two distinct axes $a,b$ of a quasi-definite 
axial algebra~$A$ of Jordan type~$1/2$ is a~primitive semisimple 
idempotent in~$A$ (Lemma~\ref{Axe}).
In~\S5, we show that if $A$ is also strongly axial, 
then $x_a(b)$ is an axis in $A$ and $A_0(a)$.

In~\S6, we prove the first of the two main results of the work:
any semisimple finite-dimensional quasi-definite strongly 
axial algebra of Jordan type $1/2$ is unital (Theorem~\ref{Thm:FD2Unit}).

In~\S7, we prove that a finitely generated quasi-definite strongly axial algebras of Jordan type $1/2$
has a finite capacity (Theorem~\ref{lengthe}).

\section{Preliminaries}

We consider only commutative (but not necessarily associative or unital) 
algebras over a~ground field $F$ of characteristic not two.
Given an element $a\in A$ and $\lambda\in F$, 
we introduce the subspace 
$A_\lambda(a) = \{x\in A\mid ax = \lambda x\}$.  
Define the map 
$\mathrm{ad}_a\colon A\to A$ as follows, $\mathrm{ad}_a(x) = ax$. 
If $\lambda\in\mathrm{Spec}(\mathrm{ad}_a)$, then $A_\lambda(a)\neq(0)$,
otherwise $A_\lambda(a) = (0)$.
An idempotent $e\in A$ is called an axis if 
$\mathrm{ad}_a$ is diagonalizable (semisimple). 
An axis $e\in A$ is called primitive if $A_1(e) = \Span\{e\}$.

A commutative algebra generated by axes is called an {\it axial algebra}. 

We will focus on primitive axial algebras of Jordan type $\eta\neq0,1$, 
it means that an axial algebra $A$ is generated by primitive axes $a_i$, $i\in I$, 
such that $\mathrm{Spec}(\mathrm{ad}_{a_i})\subset \{0,\eta,1\}$, 
and the following fusion rules are fulfilled,
\begin{equation}\label{FusionRules}
\begin{gathered}
A_0(a_i)^2\subseteq A_0(a_i),\quad
A_{\eta}(a_i)^2\subseteq A_0(a_i)+A_1(a_i),\\
(A_0(a_i)+A_1(a_i))A_{\eta}(a_i)\subseteq A_{\eta}(a_i),\quad
A_0(a_i)A_1(a_i) = (0).
\end{gathered}
\end{equation}
In particular, it implies that 
$(A_0(a_i)\oplus A_1(a_i))\oplus A_\eta(a)$ 
is a $\mathbb{Z}_2$-graded algebra.

Below, by an axial algebra of Jordan type we always mean a~primitive one. 
Analogously, an axis~$a$ in $A$ is always a primitive one.

Given an axial algebra $A$ of Jordan type~$1/2$ and an axis $a\in A$, 
the map $\tau_a\colon A\to A$ which acts as follows,
$\tau_a(x) = (-1)^{2\lambda}x$ for $x\in A_\lambda(a)$,
is an involution of $A$ called Miyamoto involution.

Axial algebras of Jordan type are known to have some significant properties.
Every axial algebra of Jordan type is spanned as a~vector space by axes~\cite{HRS}. 
Moreover, any axial algebra~$A$ of Jordan type admits 
a~unique Frobenius form, a~nonzero bilinear symmetric form 
$(\cdot,\cdot)\colon A\times A\to F$ which is invariant, i.\,e., 
$(ab,c) = (a,bc)$ for all $a,b,c\in A$, and which satisfies the property that 
$(a,a) = 1$ for every axis $a\in A$~\cite{HSS}.

Given an axial algebra of Jordan type~$A$,
$A_\mu$ and $A_\lambda$ are orthogonal with respect to the Frobenius form 
when $\lambda\neq\mu$. 
The radical of the Frobenius form 
$A^\perp := \{x \in A \mid (x, v) = 0$ for all $v \in A\}$
coincides with the unique largest ideal $R(A)$ of~$A$ 
containing no axes from~$A$~\cite{HRS}.
An axial algebra $A$ of Jordan type is called semisimple if $R(A) = (0)$.

Given an algebra~$A$, by $\langle X\rangle_{\alg}$ 
we denote the subalgebra of~$A$ generated by the set~$X$. 
The set of all words in an alphabet~$X$ we denote as $X^*$.

\begin{lemma}{(Seress Lemma, \cite{HRS})}
Given an axial algebra~$A$ of Jordan type~$1/2$ and an axis $a\in A$,
we have $a(xz) = (ax)z$ for all $x\in A$ and $z\in A_0(a)\oplus A_1(a)$.
\end{lemma}

\begin{proposition}[\cite{HRS}] \label{prop:2-gen}
Let $A = \langle a,b\rangle_{\alg}$ be an axial algebra 
of Jordan type~$1/2$ generated by two distinct axes $a, b$. 
Denote $\sigma = ab - (a+b)/2$ and $\alpha = (a,b)$.

a) $A$ is 2-dimensional precisely in the following cases:

(1) $ab = 0$, then $A = Fa\oplus Fb$ and $\alpha = 0$;

(2) $\sigma = 0$ and $\alpha = 1$.

b) $A$ is 3-dimensional precisely when $\sigma,ab\neq 0$ and 
$\sigma v = \pi v$ for $v \in \{a,b,\sigma\}$,
where $\pi = (\alpha-1)/2$.
Moreover, $A$ is unital if and only if $\alpha\neq1$, 
in which case the unit equals $\sigma/\pi$.
\end{proposition}

\begin{corollary}
Let $A = \langle a,b\rangle_{\alg}$ be an axial algebra of Jordan type~$1/2$ 
generated by two distinct axes $a, b$. 
Denote $\alpha = (a,b)$.
Then we have
\begin{equation} \label{2gen:easy-formulae}
(ab)b = \frac{1}2(\alpha b + ab),\quad
(ab)a = \frac{1}2(\alpha a + ab),\quad
(ab)(ab) = \frac{\alpha}4(a+b + 2ab).
\end{equation}
\end{corollary}

Given axes $a,b\in A$, we write
$$
b = a_0(b)+a_{1/2}(b)+\alpha(b)a,
$$
where $a_0(b)\in A_0(a)$, $a_{1/2}(b)\in A_{1/2}(a)$, 
and $\alpha(b) = (a,b)\in F$.

\begin{lemma}\label{help}
Let $A$ be an axial algebra of Jordan type $1/2$,
let $a,b\in A$ be axes, $\alpha=(a,b)$.
Then we have the following equalities,
\begin{gather}
a_0(b)^2 = (1-\alpha)a_0(b),  \label{a0^2} \\
a_{1/2}(b)^2
= \alpha a_0(b)+(\alpha-\alpha^2)a,  \label{a1/2^2} \\
a_0(b)a_{1/2}(b)
= \frac{1}{2}(1-\alpha)a_{1/2}(b). \label{a0*a1/2}
\end{gather}
\end{lemma}

\begin{proof}
Denote $a_0=a_0(b)$ and $a_{1/2}=a_{1/2}(b)$.
We have $b = a_0 + a_{1/2} + \alpha a$.
Thus,
\begin{equation} \label{b^2=b}
a_0 + a_{1/2} + \alpha a
= b
= b^2
= (a_0 + a_{1/2} + \alpha a)^2
= a_0^2 + a_{1/2}^2 + \alpha^2 a
+ 2a_0 a_{1/2} + \alpha a_{1/2}.
\end{equation}
The equality of components on $A_{1/2}(a)$ gives~\eqref{a0*a1/2}.

Also, we have $a_0=b-2ab+\alpha a$. So,
$$
a_0^2
= b+4(ab)(ab)+\alpha^2 a-4(ab)b+2\alpha ab-4\alpha(ab)a.
$$
Applying the equalities~\eqref{2gen:easy-formulae}, we derive
$$
a_0^2 = (1-\alpha)b - 2(1-\alpha)ab + \alpha(1-\alpha)a = (1-\alpha)a_0,
$$
it is exactly~\eqref{a0^2}. 
Substituting~\eqref{a0^2} in~\eqref{b^2=b} and looking at 
summands in $A_0(a)\oplus A_1(a)$, we obtain~\eqref{a1/2^2}.
\end{proof}

\begin{proposition}[{\cite[Theorem 1]{3-gen}}]\label{prop:3-gen}
Let $A = \langle a,b,c\rangle_{\alg}$ be an axial algebra 
of Jordan type~$1/2$ generated by axes $a,b,c$. 
Then $A$ is a Jordan algebra and $\dim A\leq 9$.
\end{proposition}

\begin{lemma} \label{lem:unit}
Let $A$ be an axial algebra of Jordan type $1/2$ with unit~$e$. 
Then $(e,a) = 1$ for every axis~$a$.
\end{lemma}

\begin{proof}
It follows from the properties of the Frobenius form,
$$
(e,a)
= (e,a^2)
= (ea,a)
= (a,a) 
= 1. \qedhere
$$
\end{proof}

\begin{lemma} \label{lem:radA0}
Let $A$ be an axial algebra of Jordan type $1/2$,
and let $a$ be an axis in $A$. 
Suppose that $A_0(a)$ is an axial algebra of Jordan type $1/2$,
then $R(A_0(a)) = R(A)\cap A_0(a)$.
In particular, $R(A_0(a)) = (0)$ when $A$ is semisimple.
\end{lemma}

\begin{proof}
Since the Frobenius form is defined on an axial algebra uniquely,
the restriction of the Frobenius form from $A$ on $A_0(a)$ coincides with the one defined on $A_0(a)$. 
The inclusion $R(A)\cap A_0(a)\subseteq R(A_0(a))$ is trivial.
Suppose that $x\in R(A_0(a)) = (A_0(a))^\perp$, then $x\in R(A)$, since
both $A_{1/2}(a)$ and $A_1(a)$ are orthogonal to $A_0(a)$ in~$A$.
\end{proof}

\section{Definite and quasi-definite axial algebras}

\begin{definition}
Given an axial algebra~$A$ of Jordan type and a Frobenius form $(\cdot,\cdot)$ on it,
we call $A$ anisotropic if $(x,x)=0$ implies $x=0$.
\end{definition}

Thus, an anisotropic axial algebra is semisimple.

\begin{definition}
We call an anisotropic axial algebra of Jordan type $1/2$ {\bf definite}.
\end{definition}

\begin{definition}
We call an axial algebra~$A$ of Jordan type $1/2$ {\bf quasi-definite},
if we have $(a,b)\neq1$ for every pair of distinct axes $a,b\in A$.
\end{definition}

\begin{definition}
Let $X$ be a basis of an axial algebra $A$ of Jordan type $1/2$ consisting of axes. 
We say that $X$ is a~quasi-definite basis of $A$ 
if $(x_i,x_j)\neq1$ for distinct $x_i,x_j\in X$. 
\end{definition}

The following lemma provides a~sufficient condition, 
under which an axial algebra is quasi-definite.

\begin{lemma}\label{an}
Every definite axial algebra of Jordan type $1/2$ is quasi-definite.
\end{lemma}

\begin{proof}
Denote by $A$ a~definite axial algebra of Jordan type $1/2$ 
and consider its distinct axes $a$ and~$b$.
Suppose that $(a,b)=1$.
Then by Lemma~\ref{help}, $a_{1/2}^2 = a_0$.
Thus, $(a_{1/2},a_{1/2}) = 2(a_{1/2},a_{1/2}a) = 2(a_{1/2}^2,a) 
= 2(a_0,a) = 0$.
Since $A$ is anisotropic, $a_{1/2} = 0$ and so, $a_0 = 0$.
Hence, $a = b$, a contradiction.
\end{proof}

Below we have an example of quasi-definite 
but not definite axial algebra (see~\cite{HRS}).

Let us show that a unital 3-dimensional 2-generated axial algebra~$A$ 
of Jordan type $1/2$ over a quadratically closed field~$F$ 
(it means that the equation $x^2 - a = 0$ has solutions for all $a\in F$)
of characteristic not two is quasi-definite. 
It is known that $A$ is isomorphic 
to the simple Jordan algebra of a~symmetric bilinear
non-degenerate form $f$ defined on a two-dimensional vector space~$V$,
i.\,e., $A = F1\oplus V$ with the product
\begin{equation} \label{SpinProduct}
(\alpha1+x)(\beta1+y)
= (\alpha\beta + f(x,y))1 + \alpha y + \beta x,
\end{equation}
where $\alpha,\beta\in F$ and $x,y\in V$.

Let us show that over $F$ the axial algebra $A$ is unique (up to isomorphism).
One may diagonalize a quadratic form over any field. 
Thus, there exist $e_1,e_2\in V$ such that $f(e_i,e_i) = d_i$ and $f(e_1,e_2) = 0$. 
Since $f$ is non-degenerate, we have $d_1,d_2\neq0$.
Define $u = e_1/\sqrt{d_1}$ and $v = e_2/\sqrt{d_2}$. 
Then, in the basis $1,u,v$, we have the following multiplication table:
$$
1^2 = 1, \quad
1\cdot u = u,\quad
1\cdot v = v,\quad
u^2 = v^2 = 1, \quad
uv = 0.
$$
Therefore, we get a~unique algebraic structure.
Moreover, we have a Frobenius form
$$
(1,1) = (u,u) = (v,v) = 2,\quad
(1,u) = (1,v) = (u,v) = 0.
$$

Firstly, we find all idempotents in $A$.
Let $e = \alpha1+\beta u+\gamma v$, then 
$e^2=e$ is equivalent to the relations
$$
\alpha^2 + \beta^2 + \gamma^2 = \alpha,\quad
2\alpha\beta = \beta, \quad
2\alpha\gamma = \gamma.
$$
If $\alpha=0$, then $\beta=\gamma=0$ and $e=0$.
Otherwise, either $e = 1$ or $\alpha=1/2$ and $\beta^2 + \gamma^2=1/4$.

Let us check that a non-unital idempotent $e$ is a primitive one.
Indeed, $A_0(e) = \Span\{1/2-(\beta u+\gamma v)\}$,
$A_{1/2}(e) = \Span\{-\gamma u + \beta v\}$, and $A_1(e) = \Span\{e\}$.

Note that the definition of the Frobenius form is correct, since
for every primitive idempotent~$e$, we have
$$
(e,e)
= (1/2+\beta u+\gamma v,1/2+\beta u+\gamma v)
= 1/2 + 2(\beta^2+\gamma^2) = 1/2+1/2=1.
$$

Suppose that $(e,e') = 1$, where $e = 1/2+\beta u+\gamma v$
and $e' = 1/2+\delta u+\varepsilon v$ are distinct axes. 
Then we check that
$$
(e,e')
= (1/2+\beta u+\gamma v,1/2+\delta u+\varepsilon v)
= 1/2 + 2(\beta\delta+\gamma\varepsilon) = 1,
$$
so, $\beta\delta+\gamma\varepsilon = 1/4$.
Thus, we have
$\begin{pmatrix}
\beta & \gamma \\
\delta & \varepsilon \\
\end{pmatrix}w
= \begin{pmatrix}
0 \\
0
\end{pmatrix}$, 
where $w = (\beta-\delta,\gamma-\varepsilon)^T$
in the basis $u,v$. 
If $e\neq e'$, then
the matrix is degenerate.
Suppose that
$\delta u+\varepsilon v = k(\beta u+\gamma v)$ for some $k\in F$.
Therefore,
$$
1/4
= \beta\delta+\gamma\varepsilon
=k(\beta^2+\gamma^2) = k/4.
$$
It means that $k=1$ and $e = e'$.

However, $A$ is not definite, e.\,g.,
$(1+iu,1+iu) = 0$.

Now, consider the case of $(n+1)$-dimensional Jordan algebra 
$J = F1\oplus V$
of Clifford type (also known as a spin factor) 
with the product defined by~\eqref{SpinProduct}.
Let us restrict the conditions as follows, we assume that $F = \mathbb{R}$
and the form $f$ is positive-definite. 
Then we may find a~linear basis $v_1,\ldots,v_n$ of~$V$ such that 
the Frobenius form satisfies
$(1,1) = (v_1,v_1) = \ldots = (v_n,v_n) = 2$ and 
$(1,v_i) = (v_i,v_j) = 0$ for all $i,j\in\{1,\ldots,n\}$ and $i\neq j$.
Thus, $a_i = (1+v_i)/2$, $i=1,\ldots,n$, are axes generating~$J$
and satisfying $(a_i,a_i) = 1$.
For every $x\in J$, we have a representation 
$x = \alpha1 + \sum\limits_{i=1}^n \beta_i v_i$.
Then 
$(x,x)/2 = \alpha^2 + \sum\limits_{i=1}^n \beta_i^2$.
If $(x,x) = 0$, then $\alpha = \beta_1 = \ldots = \beta_n = 0$ and $x = 0$.
Hence, $J$ is a definite axial algebra of Jordan type~$1/2$.
Note that there is no a contradiction to the previous example, 
since the field $\mathbb{R}$ is not quadratically closed.

Denote by $M_n^{(+)}(F)$ the matrix algebra $M_n(F)$ 
under the product $A\circ B = (AB + BA)/2$.
It is well-known that $M_n^{(+)}(F)$ is a simple Jordan algebra.
Thus, it is an axial algebra of Jordan type~$1/2$.
Moreover, its Frobenius form coincides with the trace one, i.\,e.
$(X,Y) = \tr(X\circ Y) = \tr(XY)$ for all $X,Y\in M_n^{(+)}(F)$.

The matrix unity is denoted by $e_{ij}$, $1\leq i,j\leq n$.

\begin{lemma} \label{lem:axisInMatrix}
A matrix $X\in M_n^{(+)}(F)$ is a primitive axis 
if and only if $X^2 = X$ and $\rank X = 1$.
\end{lemma}

\begin{proof}
Let $X$ be an idempotent matrix from $A = M_n^{(+)}(F)$.
Then $X$ is conjugate to the matrix $e_{11}+\ldots+e_{rr}$, 
where $r = \rank X$.
Note that 
$A_1(X) = \{S\in M_n^{(+)}(F)\mid X\circ S = S\} 
= \Span\{e_{11},\ldots,e_{rr}\}$.
So, $X$ is primitive if and only if $r = 1$.
\end{proof}

For $n\geq2$, $M_n^{(+)}(F)$ is not quasi-definite.
Indeed, consider primitive axes $X = e_{11}$ and $Y = e_{11} + e_{12}$,
then $(X,Y) = 1$. On the other hand, 
$M_n^{(+)}(F)$ has a~quasi-definite basis of axes.

Given $a = (a_1,\ldots,a_n)\in F_n$ and $b = (b_1,\ldots,b_n)\in F_n$,
introduce $\langle a,b\rangle = \sum\limits_{j=1}^n a_jb_j$.

\begin{proposition}
Let $F$ be an infinite field. Then
$M_n^{(+)}(F)$ is an axial algebra of Jordan type~$1/2$
which has a quasi-definite basis of axes.
\end{proposition}

\begin{proof}
We prove the statement by induction on~$n$.
For $n = 1$, we take the matrix $e_{11}$ as the required basis.

Suppose that we have found quasi-definite basis for $M_n^{(+)}(F)$
consisting of matrices $H_1,\ldots,H_{n^2}$.
By Lemma~\ref{lem:axisInMatrix}, since $\rank H_i = 1$, 
we may present $H_i = l_i^T r_i$, where 
$l_i = (l_{i1},\ldots,l_{in})$ and
$r_i = (r_{i1},\ldots,r_{in})$ are vectors from $F_n$.
Then $H_1,\ldots,H_{n^2}$ form a~quasi-definite basis of axes
if and only if 
$$
\langle l_i,r_i\rangle = 1,\ i=1,\ldots,n^2, \quad
\langle l_i,r_j\rangle\langle l_j,r_i\rangle \neq 1, \ 
i\neq j,\ i,j\in\{1,\ldots,n^2\}.
$$

Now consider the algebra $A = M_{n+1}^{(+)}(F)$.
We identify matrices $H_i$ with their images in $A$ under the embedding
$\psi\colon M_n^{(+)}(F)\to M_{n+1}^{(+)}(F)$ defined as follows,
$\psi(e_{ij}) = e_{ij}$, $i,j=1,\ldots,n$.

Introduce matrices 
$$
F_i(b_i,d_i) 
= (1-d_i)e_{ii} + b_i e_{i\,n+1}
+ \frac{d_i(1-d_i)}{b_i}e_{n+1\,i} + d_i e_{n+1\,n+1},\quad i=1,\ldots,n. 
$$
By Lemma~\ref{lem:axisInMatrix}, they are all primitive axes.

We want to find $d_i,b_i,c_i\in F$ for $i=1,\ldots,n$ such that
$$
S = \{H_1,\ldots,H_{n^2},F_1(b_1,d_1),F_1(c_1,d_1),
\ldots,F_n(b_n,d_n),F_n(c_n,d_n),e_{n+1\,n+1}\}
$$
is a required quasi-definite basis of $A$. 

For each $i=1,\ldots,n$, we take nonzero $b_i,c_i$ such that 
$b_i\neq\pm c_i$ and $d_i\neq0,1$.
The inductive hypothesis implies that the set~$S$ 
of primitive idempotents forms a~basis of~$A$.

It remains to check that $(X,Y)\neq1$ for all $X,Y\in S$, $X\neq Y$.
By the inductive hypothesis $(H_i,H_j)\neq1$ when $i\neq j$.
Note that $(H_i,e_{n+1\,n+1}) = 0$.
We have $(F_i(b_i,d_i),e_{n+1\,n+1}) 
= (F_i(c_i,d_i),e_{n+1\,n+1}) = d_i\neq1$.

Further, we have
$$
(H_i,F_j(x,d_j))
= (1-d_j)l_{ij}r_{ij},
$$
where $x = b_i$ or $x = c_i$,
for all $i=1,\ldots,n^2$ and $j=1,\ldots,n$.
Thus, we take $d_j$ such that 
$1/(1-d_j)\not\in K_j = \{l_{ij}r_{ij}\mid i=1,\ldots,n^2\}$.

Now, we compute
$$
(F_i(x,d_i),F_j(y,d_j))
= d_i d_j,\quad i\neq j,
$$ 
where $x \in \{b_i,c_i\}$ and $y\in \{b_j,c_j\}$.
Thus, we take $d_k$ in such way that $d_i d_j\neq 1$ 
for all $i,j=1,\ldots,n$.

Finally, we get for $i=1,\ldots,n$,
\begin{multline*}
(F_i(b_i,d_i),F_i(c_i,d_i))
= \tr\left( 
\begin{pmatrix}
	1-d_i & b_i \\
	\frac{d_i(1-d_i)}{b_i} & d_i 
\end{pmatrix}
\begin{pmatrix}
	1-d_i & c_i \\
	\frac{d_i(1-d_i)}{c_i} & d_i 
\end{pmatrix}
\right) \\
= (1-d_i)^2 + d_i^2 + d_i(1-d_i)\left(\frac{b_i}{c_i}+\frac{c_i}{b_i}\right)
= 1 + d_i(1-d_i)\left(\frac{b_i}{c_i}+\frac{c_i}{b_i}-2\right)\neq1
\end{multline*}
when $b_i\neq c_i$.

Summarizing, we choose nonzero $b_i,c_i$, $i=1,\ldots,n$, such that $b_i\neq \pm c_i$.
Also, we choose $d_i\not\in \{0,1\}$ and $1/(1-d_i)\not\in K_i$ and, 
moreover, $d_id_j\neq1$ for all $i,j=1,\ldots,n$.
We may do it, since $F$ is infinite.
\end{proof}

Let $n\geq2$, denote by $H_n(F)$ the set of 
all symmetric matrices of order~$n$ over~$F$.
The space $H_n(F)$ under the product $A\circ B = (AB + BA)/2$ 
is also a~simple Jordan algebra.
By~\cite[Theorem 3.4]{MR}, its Jordan subalgebra $H_n'(F)$ 
consisting of all symmetric matrices with zero row sum 
has a~quasi-definite basis, since $H_n'(F)$ is isomorphic 
to the Matsuo algebra $M_{1/2}(G,D)$ for corresponding group~$G$ 
and set of involution $D$.

Let us recall the definition of Matsuo algebra.
Given a~group~$G$ generated by a set $D$ of involutions,
the Matsuo algebra $M_{\eta}(G,D)$, where $\eta\neq0,1$, 
is a~vector space~$\Span\{D\}$ with the product
$$
c\cdot d
= \begin{cases}
c, & |cd| = 1, \\
0, & |cd| = 2, \\
\frac{\eta}{2}(c+d-c^d), & |cd| = 3.
\end{cases}
$$
The Frobenius form on~$M_{\eta}(G,D)$ is defined as follows,
\begin{equation} \label{MatsuoForm}
(c,d)
= \begin{cases}
1, & c = d, \\
0, & |cd| = 2, \\
\frac{\eta}{2}, & |cd| = 3.
\end{cases}
\end{equation}
If $\eta = 1/2$, then $M_{1/2}(G,D)$ is an axial algebra of Jordan type~$1/2$.
By~\eqref{MatsuoForm}, $M_{1/2}(G,D)$ has a~quasi-definite basis.
If the automorphism group of $M_{1/2}(G,D)$ coincides 
with the Miyamoto group, then $M_{1/2}(G,D)$ is quasi-definite.

For the algebra~$H_n(F)$ over the field of real numbers, 
we may say much more.

\begin{proposition}
Algebra $H_n(\mathbb{R})$ is a definite axial algebra of Jordan type~$1/2$.
\end{proposition}

\begin{proof}
For the Frobenius form $(X,Y) = \tr(XY^T)$, we have 
the Cauchy---Bunyakovsky---Schwarz inequality
$\tr(XY^T)^2\leq \tr(XX^T)\tr(YY^T)$.
Applying it for pairwise distinct axes~$A$~and~$B$, we get
$(A,B)^2 \leq (A,A)(B,B) = 1$, and we have the equality 
if and only if $A$ and $B$ are linearly dependent. 
We conclude that $A = B$, a contradiction.
\end{proof}

\begin{remark}
For $n=2$, the statement holds over every field, 
since $H_2(F)$ coincides with 3-dimensional simple Jordan algebra 
of Clifford type considered above.
\end{remark}

\section{Idempotent $x_a(b)$ in quasi-definite axial algebra}

Given an axial algebra~$A$ of Jordan type $1/2$ 
and distinct axes $a,b\in A$, we introduce the element
\begin{equation} \label{formula:xa(b)}
x_a(b) = \frac{2ab - (a,b)a - b}{(a,b)-1} = \frac{a_0(b)}{1-(a,b)}.
\end{equation}
It is clear that element $x_a(b)$ is defined only if $(a,b)\neq 1$. 
By Lemma~\ref{help}, $x_a(b)$ is nonzero.

\begin{lemma}\label{id}
Let $A$ be an axial algebra of Jordan type $1/2$, 
and let $a,b\in A$ be axes such that $(a,b)\neq1$. 
Then $x_a(b)$ is an idempotent, 
$a+x_a(b)$ is a~unit in $\langle a,b\rangle_{\alg}$,
and $(x_a(b),x_a(b)) = 1$.
\end{lemma}

\begin{proof}
If $\dim(\langle a,b\rangle_{\alg}) = 3$, then 
by Proposition~\ref{prop:2-gen}b, 
$\sigma/\pi$ is a unit of $\langle a,b\rangle_{\alg}$.
Thus, $x_a(b) = \sigma/\pi-a$ is an idempotent.
If $\dim(\langle a,b\rangle_{\alg}) = 2$, then 
$\langle a,b\rangle_{\alg} = Fa\oplus Fb$ 
by Proposition~\ref{prop:2-gen}a. Thus, $x_a(b) = b$.

Denote $\alpha = (a,b)$ and calculate by Lemma~\ref{help}:
\begin{multline*}
(x_a(b),x_a(b)) 
 = \frac{1}{(1-\alpha)^2}(a_0(b),a_0(b)) \\
 = \frac{1}{(1-\alpha)^2}((a_0(b)+a_{1/2}(b)+\alpha a,a_0(b)
 + a_{1/2}(b)+\alpha a)-(a_{1/2}(b),a_{1/2}(b)-\alpha^2(a,a)) \\
 = \frac{1}{(1-\alpha)^2}(1-\alpha^2-2(a_{1/2},a_{1/2}a))
 = \frac{1}{(1-\alpha)^2}(1-\alpha^2-2(a^2_{1/2},a))
 = 1.
\qedhere
\end{multline*}
\end{proof}

\begin{lemma}\label{axe3gen}
Let $A$ be an axial algebra of Jordan type $1/2$.
Let $a,b,c$ be axes such that $(a,b)\neq 1$ and 
$C = \langle a,b,c\rangle_{\alg}$.
Then $x_a(b)$ is a~primitive axis in $C$. 
\end{lemma}

\begin{proof}
By Lemma~\ref{id}, $x_a(b)$ is an idempotent in~$C$.
By Proposition~\ref{prop:3-gen}, $x_a(b)$ 
is a~semisimple idempotent in~$C$ 
with the fusion rules~\eqref{FusionRules} fulfilled. 
So, it remains to check that $x_a(b)$ is a~primitive idempotent. 

When $C$ is the universal 3-generated 9-dimensional axial algebra~$U_3$ 
of Jordan type $1/2$, it follows by~\cite[Code\#1]{github}.
The fact that $x_a(b)$ is a primitive axis 
in every quotient of $U_3$ by the ideal~$I$ 
not containing $x_a(b)$ follows from the property that 
$I$~contains all homogeneous components of its elements. 
Indeed, let $y\in U_3$ be such that $x_a(b)y-y\in I$, i.\,e., 
$x_a(b)y = y$ in $U_3/I$. 
Since $y = y_0+y_{1/2}+\alpha x_a(b)$ with respect 
to the idempotent $x_a(b)$, we get
$$
x_a(b)y - y = -(y_0+y_{1/2}/2)\in I.
$$
Thus, $y_0,y_{1/2}\in I$ and $y\in \Span\{x_a(b)\} + I$, 
it means that $x_a(b)$ is primitive in $U_3/I$.
\end{proof}

\begin{lemma}\label{Axe}
Let $A$ be an axial algebra $A$ of Jordan type $1/2$.
Let $a,b$ be two distinct axes such that $(a,b)\neq1$. 
Then $c = x_a(b)$ is a primitive semisimple idempotent in~$A$,
and eigenvalues of $\ad_c$ lie in the set $\{0,1/2,1\}$.
\end{lemma}

\begin{proof}
By Lemma~\ref{id}, $c = x_a(b)$ is an idempotent.

Let $B$ be a linear basis of $A$ consisting of axes.
We may assume that $a,b\in B$. Let us state that
\begin{equation} \label{formula:x_a(b)-axis}
A = Y_0(c) + Y_{1/2}(c) + \Span\{c\},
\end{equation}
where $Y_{\varepsilon}(c) = \{z\in A \mid cz = \varepsilon z\}$.
By Lemma~\ref{axe3gen}, for every $r\in B$ we have
$r = r_0(c)+r_{1/2}(c)+\beta_r x_a(b)$.
Here $r_0(c)\in Y_0(c)$, $r_{1/2}(c)\in Y_{1/2}(c)$, and $\beta_r\in F$. 
Since each element of the basis has a~decomposition 
on the eigensubspaces of $\ad_{x_a(b)}$, the statement follows. 
\end{proof}

Suppose that $B$ is a basis of an axial algebra~$A$ of Jordan type $1/2$ 
consisting of axes such that $(a,y)\neq1$ for a fixed $a\in B$ and every 
$y\in B\setminus\{a\}$. Define $B_0 = \{x_a(y) \mid y\in B\setminus\{a\}\}$.

\begin{lemma}\label{bazis}
Let $A$ be an axial algebra $A$ of Jordan type $1/2$, 
and let $A$ have a~quasi-definite basis~$B$.
For $a\in B$, we have $\Span \{B_0\} = A_0(a)$.
\end{lemma}

\begin{proof}
Let $z\in A_0(a)$, we may present it as 
a~linear combination of elements from~$B$,
$$
z = \sum\limits_{j\in J}\varkappa_j b_j
= \sum\limits_{j\in J_1}\varkappa_j b_j
+ \sum\limits_{j\in J_2}\varkappa_j b_j,
$$
where $J = J_1 \dot \cup J_2$ and
$j\in J_2$ if and only if $a_0(b_j) = 0$.
Also, we present 
$b_j = a_0(b_j) + a_{1/2}(b_j) + (a,b_j)a$ for every $b_j\in B$.

For $j\in J_1$, we have $a_0(b_j) = (1-(a,b_j))x_a(b_j)$.
Since $A$ is a direct vector-space sum of its subspaces 
$A_0(a)$, $A_{1/2}(a)$, and $\Span \{a\}$, $z$ lies in $\Span\{B_0\}$.
\end{proof}

Let $a,b,c$ be axes in an axial algebras of Jordan type~$1/2$ 
such that $(a,b),(a,c),(b,c)\neq1$.
The following example shows that it may happen 
$(x_a(b),x_a(c))=1$ when $x_a(b)\neq x_a(c)$.

\begin{example}
Let $A = A(\alpha,\beta,\gamma,\phi)$ be the universal 
3-generated axial algebra of Jordan type~$1/2$. 
Assume that $\alpha = (a,b)$, $\beta = (b,c)$ 
and $\gamma = (a,c)$ are not equal to 1. 
For the axes $x_a(b)$ and $x_a(c)$, 
we have by~\cite[Code\#2]{github}
$$
(x_a(b),x_a(c))
= \frac{-\alpha\gamma-\beta+2\phi}{-\alpha\gamma+\alpha+\gamma-1},
$$
which is equal to 1 if and only if $\alpha+\beta+\gamma-2\phi-1=0$.
However, $x_a(b)\neq x_a(c)$.
\end{example}

\begin{lemma}\label{xa(b)PrimInA0}
Let $A$ be an axial algebra of Jordan type $1/2$. 
Let $a,b$ be axes of~$A$ such that $(a,b)\neq0,1$.
Then $A_0(a)\cap A_{1/2}(b) = (0)$.
\end{lemma}

\begin{proof}
Put $\alpha = (a,b)$.
For $z\in A_0(a)\cap A_{1/2}(b)$, we compute
$$
z/2 = zb 
= z(a_0(b)+a_{1/2}(b)+\alpha a)
= za_0(b) + za_{1/2}(b).
$$
Since $za_0(b) \in A_0(a)$, $za_{1/2}(b) \in A_{1/2}(a)$, 
we get $za_{1/2}(b) = 0$ and $zx_a(b) = z/(2-2\alpha)$.
By~Lemma~\ref{Axe}, $\ad_{x_a(b)}$ has eigenvalues 
only from the set $\{0,1/2,1\}$. 
Since $1/(2-2\alpha)\neq0$ and $1/(2-2\alpha)\neq1/2$ when $\alpha\neq0$,
we have the only case $\alpha = 1/2$ and $z\in A_1(x_a(b))$.
Since $x_a(b)$ is primitive, 
we derive that $z = kx_a(b)$ for some $k\in F$.
If $k\neq0$, then $(x_a(b))^2 = x_a(b)$ lies in 
$A_{1/2}(b)$ and in $A_{1/2}(b)^2\subset A_0(b)+\Span(b)$ 
at the same time. Thus, $z = 0$.
\end{proof}

\section{Subalgebra $A_0(a)$ in quasi-definite strongly axial case}

\begin{definition}
We call an axial algebra~$A$ of Jordan type $1/2$ {\bf strongly axial}, 
if every primitive semisimple idempotent with eigenvalues $\{0,1/2,1\}$ 
in~$A$ is a~primitive axis.
\end{definition}

Every axial algebra of Jordan type $1/2$ 
which is also a~Jordan one is strongly axial~\cite{Albert}.

\begin{lemma}\label{axisA2B}
Let $A$ be an axial algebra of Jordan type, let $B$ be a subalgebra of~$A$,
and let $a\in B$ be an axis in~$A$. Then $a$ is an axis in~$B$.
\end{lemma}

\begin{proof}
The equality $B = B_0(a)\oplus B_{1/2}(a) \oplus \Span\{a\}$ holds 
due to the fact that $\ad_a$ is a semisimple operator on $A$ 
with the eigenvalues $\{0,1/2,1\}$. Since  
$B_0(a) = B\cap A_0(a)$, $B_{1/2}(a) = B\cap A_{1/2}(a)$,
the fusion rules~\eqref{FusionRules} are automatically fulfilled in~$B$.
\end{proof}

\begin{proposition}\label{AxA0}
Let $A$ be a quasi-definite strongly axial algebra of Jordan type $1/2$ 
and let $B$ be a subalgebra of~$A$.
Let $a,b\in B$ be distinct axes in $A$, then 
$x_a(b)$ is a primitive axis in $A$, $B$, and in $A_0(a)\cap B$.
\end{proposition}

\begin{proof}
By Lemma~\ref{Axe}, $x_a(b)$ is a~primitive semisimple idempotent 
with eigenvalues $\{0,1/2,1\}$ in~$A$.
Since $A$ is strongly axial, then $x_a(b)$ is a~primitive axis in~$A$. 

By Lemma~\ref{axisA2B}, $x_a(b)$ is an axis in~$B$ and $A_0(a)\cap B$ too.
\end{proof}

\begin{corollary} \label{A0-QDBasis}
Let $A$ be a quasi-definite strongly axial algebra of Jordan type $1/2$, 
and let $a$ be an axis in $A$, and $\dim A>1$.
Then $A_0(a)$ is an axial algebra of Jordan type $1/2$ 
which contains a quasi-definite basis.
\end{corollary}

\begin{proof}
Let $X$ be a basis of~$A$ consisting of axes. 
From Lemma~\ref{bazis} it follows that 
$X_0=\{x_a(x_i)\mid x_i\in X\}$ includes a basis of $A_0(a)$
which by Proposition~\ref{AxA0} consists of axes in $A_0(a)$.
Since all $x_a(x_i)$ are axes in $A$, this basis is quasi-definite. 
\end{proof}

\begin{lemma}\label{zeroaxis}
Let $A$ be a quasi-definite strongly axial algebra, and let $a$~and~$b$ be axes. 
We have $(a,b) = 0$ if and only if $b\in A_0(a)$.
\end{lemma}

\begin{proof}
If $b\in A_0(a)$, then $(a,b) = 0$ by definition.

Assume that $b\not\in A_0(a)$. Since $(a,b)=0$, we have $x_a(b) = a_0(b)$. 
By Proposition~\ref{AxA0}, $a_0(b)$ is an axis of $A$. 
Therefore, applying orthogonality of distinct eigenspaces 
due to the Frobenius form, we compute
$$
(a_0(b),b)
= (a_0(b),a_0(b) + a_{1/2}(b) + \alpha a)
= (a_0(b),a_0(b))
= 1.
$$
Hence, we get a~contradiction with the definition of quasi-definite axial algebra. 
\end{proof}

\begin{lemma}\label{lem:SmallerWordsToAxis}
Let $A$ be a quasi-definite strongly axial algebra of Jordan type $1/2$ 
with a~generating set $X$ of axes.
Let $w \in X^*\setminus\{1\}$.
Then there exists a~linear combination~$s$ of words of the length 
less than $|w|$ such that $q = w+s$ is (up to a~nonzero scalar) an axis of $A$.
\end{lemma}

\begin{proof}
Let us prove by induction on $|w|$. If $|w| = 1$, then the statement is trivial. 
When $w = ab$ for $a,b\in X$, it follows by the formula~\eqref{formula:xa(b)}, 
where we take $q = x_a(b)$.

Suppose that $w = w_1a$ for some $w_1\in X^*$ and $a\in X$.
The induction hypothesis states that there exist a~linear combination~$s_1$ 
of words of the length less than $|w_1|$ and a nonzero scalar~$\alpha$ 
such that $q_1 = \alpha(w_1+s_1)$ is an axis of $A$.
Thus, the following equality
$w+s_1a = (w_1+s_1)a = (1/\alpha)q_1a$ 
reduces the general case to the already studied case when $|w| = 2$.
\end{proof}

\section{Sufficient conditions to contain unit}

\begin{lemma}\label{ed1}
Let $A$ be an axial algebra of Jordan type $1/2$, and let $a$ be an axis in~$A$
such that $A_0(a)$ contains a~unit $e_0(a)$. 
Then we have $(e_0(a)+a,b)=1$ for every axis $b\in A$ satisfying 
the condition that $(a,b)\neq1$.
\end{lemma}

\begin{proof}
Denote $\alpha=(a,b)$, $a_0=a_0(b)$, $a_{1/2}=a_{1/2}(b)$.
Consider $e = e_0(a) + a$ and compute 
$$
eb 
= (e_0+a)(a_0+a_{1/2}+\alpha a)
= a_0 + e_0a_{1/2} + (1/2)a_{1/2} + \alpha a
= b - (1/2)a_{1/2} + e_0a_{1/2}.
$$ 
Further, using pairwise orthogonality of $A_0(a)$, $A_{1/2}(a)$, 
and $A_1(a)$, we deduce
\begin{equation} \label{(e,b)=1}
(e,b)=(eb,b)=(b,b - (1/2)a_{1/2} + e_0a_{1/2}) 
= 1 - (1/2)(a_{1/2},a_{1/2}) + (a_{1/2},e_0 a_{1/2}).
\end{equation}

We apply Lemmas~\ref{help} and~\ref{id} and the condition that $(a,b)\neq1$ to compute
\begin{multline*}
(a_{1/2}, e_0a_{1/2})
= (a_{1/2}^2,e_0)
= (\alpha a_0 + \alpha(1-\alpha)a,e_0)
= \alpha(a_0,e_0) \\
= \alpha(1-\alpha)(x_a(b),e_0)
= \alpha(1-\alpha)(x_a(b)^2,e_0)
= \alpha(1-\alpha)(x_a(b),x_a(b))
= \alpha(1-\alpha), 
\end{multline*}
$$
(a_{1/2},a_{1/2})
= 2(aa_{1/2},a_{1/2})
= 2(a,a_{1/2}^2)
= 2(a,\alpha a_0+\alpha(1-\alpha)a)
= 2\alpha (1-\alpha).
$$
Hence, $(e,b) = 1$.
\end{proof}

\begin{lemma}\label{ed}
Let $A$ be a semisimple axial algebra of Jordan type $1/2$ 
with a~quasi-definite basis~$X$. 
Suppose that $A_0(x_i)$ contains a unit $e_0(x_i)$ for each $x_i\in X$. 
Then $e=e_0(x_1) + x_1$ is a~unit of~$A$.
\end{lemma}

\begin{proof}
Let $a,b\in X$ be distinct.
Denote $e = e_0(a)+a$ and $\tilde{e} = e_0(b)+b$.
From Lemma~\ref{ed1} it follows that 
$(e,x_i)=(\tilde{e},x_i) = 1$ for each $x_i\in X$.

Let $r=e-\tilde{e}$. 
For each $x_i\in X$, we have $(r,x_i)=0$. 

Since~$X$ is a linear basis (maybe, infinite) consisting of axes, 
$e-\tilde{e}$ lies in the radical $A^\perp$ of the Frobenius form. 
By $A^\perp = (0)$, we get $e = \tilde{e}$. Hence,
$$
eb = \tilde{e}b
= (e_0(b)+b)b = b^2 = b.
$$
It holds for every axis~$b\in X$, thus, $e$ is a unit of $A$.
\end{proof}

\begin{theorem} \label{Thm:FD2Unit}
Let $A$ be a semisimple finite-dimensional quasi-definite strongly axial 
algebra of Jordan type $1/2$. Then $A$ is unital.
\end{theorem}

\begin{proof}
Denote $N = \dim A$.
Let $C_0$ be a~basis of $A$ consisting of axes. 
Fix $a_0\in C_0$, by Corollary~\ref{A0-QDBasis}, 
one may find a~quasi-definite basis 
$C_1 \subset \{x_{a_0}(y)\mid y\in C_0\setminus\{a_0\}\}$ 
of an axial algebra $A_0(a_0)$, which consist of axes of~$A$ by Proposition~\ref{AxA0}.
Let us indicate that $A_0(a_0)$ may be not necessarily quasi-definite and strongly axial algebra.
However, we do not need such conditions!
Throughout the proof, we consider only elements of the form $x_y(z)$, 
where $y,z$ are axes in~$A$, thus, Proposition~\ref{AxA0} 
implies that all such elements are axes in~$A$ as well as 
in all required axial subalgebras of~$A$.
Note that $A_0(a_0)$ is semisimple by Lemma~\ref{lem:radA0}.

Now, we choose $a_1 \in C_1$, by Corollary~\ref{A0-QDBasis}
one may find a~quasi-definite basis 
$C_2 \subset \{x_{a_1}(y)\mid y\in C_1\setminus\{a_1\}\}$ 
of an axial algebra $A_0(a_0)$, which consist of axes of~$A$ again by Proposition~\ref{AxA0}.
We denote $(A_0(a_0))_0(a_1)$ as $A_0(a_0,a_1)$.

Continue on, since the choice of $a_0,\ldots,a_n,\ldots$
is not unique, we get a finite set of algebras of the form 
$A_0(a_0,\ldots,a_n)$, where $n<N$ and $A_0(a_0,\ldots,a_n)$ 
is defined by induction as $(A_0(a_0,\ldots,a_{n-1}))_0(a_n)$. 
Now, we prove by induction on the dimension that all of them are unital. 
If $\dim A_0(a_0,\ldots,a_n) = 1$, then $A_0(a_0,\ldots,a_n)\cong F$ and hence is unital.

Suppose it is proved that $A_0(a_0,\ldots,a_n)$ is unital if 
$\dim A_0(a_0,\ldots,a_n)<k$. 
Consider the case when $\dim Y = k$, where 
$Y = A_0(a_0,\ldots,a_n)$. 
Applying Lemma~\ref{ed} and the induction hypothesis, 
we conclude that $Y$~is unital. 
Therefore, by induction, $A_0(a)$ is unital for every $a\in C_0$. 
Then $A$ is unital by Lemma~\ref{ed}.
\end{proof}

\section{Finite capacity of unital algebras}

\begin{definition}
Let $e$ be a unit of an axial algebra $A$ of Jordan type $1/2$.
Suppose that $e$~equals a sum of $n$~pairwise orthogonal axes 
and $n$~is minimal satisfying such property. 
We say that $e$ has the capacity~$n$ and write it as $c(e)=n$. 
\end{definition}

For the simple Jordan algebra $A = F1 + Fu + Fv$ of a nondegenerate form 
defined above, we have
$e = (1/2 + (u+v)/\sqrt{2}) + (1/2 - (u+v)/\sqrt{2})$,
it is a decomposition  of $e$ into a sum of two axes.
Thus, $c(e) = 2$, and all decompositions of~$e$ into 
a~finite sum of pairwise orthogonal axes consist of only two summands.

Something similar we have with the algebra~$M_n^{(+)}$, now $c(e) = n$, 
and there is a~natural decomposition $e = e_{11} + \ldots + e_{nn}$.
Note that the decomposition is not unique, e.\,g., 
$e = (e_{11} + e_{12}) + (e_{22} - e_{12})$ for $n=2$.

\begin{lemma}\label{predUnitProperty}
Let $A$ be a quasi-definite strongly axial algebra of Jordan type $1/2$. 
Let $q,a,b$ be axes in~$A$. 
Suppose that $qa = qx_a(b) = 0$. Then $qb = 0$.
\end{lemma}

\begin{proof}
Denote $\alpha =(a,b)$.
By Lemma~\ref{zeroaxis}, we prove the statement,
$$
0 = (\alpha -1)(q,x_a(b))
= (q, 2ab-\alpha a - b)
= 2(q,ab) - (q,b)
= 2(qa,b) - (q,b)
= - (q,b). \qedhere
$$
\end{proof}

\begin{theorem}\label{lengthe}
Let $A$ be a quasi-definite strongly axial algebra of Jordan type $1/2$ and 
let $G = \{a_1,\ldots,a_n\}$ be a generating set of $A$ consisting of axes. 
Let $e$ be a unit of $A$. 
Then $e$ has a finite capacity $k\leq n$.
\end{theorem}

\begin{proof}
Given a subalgebra $B$ of $A$ and $x\in B$, denote the eigensubspace 
of the operator $\ad_x\colon B\to B$ corresponding to the eigenvalue~0 by $A_0(x,B)$. 
Note that $e_2=e-a_1$ is a~unit of the algebra $A_2=A_0(a_1, A)$. Denote 
\begin{equation} \label{G2}
G_2(a_1) = \{b_1=x_{a_1}(a_2),\ldots,b_{n-1}=x_{a_1}(a_n)\}.
\end{equation}
By Corollary~\ref{A0-QDBasis}, the set $G_2$ contains a subset 
of the~quasi-definite basis of $A_2$. 
On the next step, we fix $b_1$ and project the remaining axes 
on the subalgebra $A_3 = A_0(b_1,A_2)$:
\begin{equation} \label{G3}
G_3(a_1,b_1) = \{x_{b_1}(b_2),\ldots,x_{b_1}(b_{n-1})\},
\end{equation}
and $e_3 = e_2 - b_1 = e - a_1 - b_1$ is a unit of~$A_3$.
By induction, on the last step, we get the set 
$G_k(a_1,b_1,\ldots,c_1)=\{d_1=x_{c_1}(c_2)\}$.

Define $e_{k+1} = e_k - d_1 = e - (a_1+b_1+\ldots+c_1+d_1)$.
If $e_{k+1} = 0$, then $e$ can be represented in the required form.

Suppose that $e_{k+1}\neq0$. 
Then by Corollary~\ref{A0-QDBasis}, $A_{k+1} = A_0(d_1,A_k)$ 
is an axial algebra with a~basis of axes from~$A$. 
Take~$q$ such an axis.
Then, by the definition, $qd_1=0$ and $qc_1 = 0$. 
By Lemma~\ref{predUnitProperty}, $qc_2=0$. 

Suppose that $c_1 = x_{f_1}(f_2)$ and $c_2 = x_{f_1}(f_3)$.
By the definition, $qf_1 = 0$. Since $qc_2 = 0$, then again 
by Lemma~\ref{predUnitProperty}, $qf_2 = 0$.
Continuing the procedure, we get $qa_i=0$ for every $1\leq i\leq n$. 

If $qA = 0$, then $q\cdot q = q = 0$, a contradiction.

Suppose that $qA \neq 0$. 
Let $w$ be a word in the alphabet~$G$ of a~minimal length 
such that $qw\neq 0$. Let $w = w_1w_2$. 
By the assumption, $qw_1 = qw_2 = 0$. 
Let $s$ be a~linear combination of words of the length less than $|w_1|$ 
and $\alpha\in F\setminus\{0\}$ such that $\alpha(w_1+s) = r$ 
for some axis $r$ (see Lemma~\ref{lem:SmallerWordsToAxis}). 
By the assumption, we have $qr = q(sw_2) = 0$. 
By Seress Lemma,
$$
\alpha q(w_1w_2) 
= \alpha q(w_1w_2) \pm \alpha q(sw_2)
= q(rw_2)
= r(qw_2)
= 0,
$$
a contradiction.
\end{proof}

Let $A$ be a finitely generated quasi-definite strongly axial 
algebra of Jordan type~$1/2$ and let $G = \{a_1,\ldots,a_n\}$ 
be a generating set of~$A$ consisting of axes. 
Moreover, let $y_1,\ldots,y_s\in A$ be axes such that $y_1\in G$,
$y_2\in G_2(y_1)$, \ldots, $y_s\in G_s(y_1,\ldots,y_{s-1})$,
where the sets $G_i(y_1,\ldots,y_{i-1})$ 
are defined by~\eqref{G2},~\eqref{G3} etc. 
Denote the subalgebra 
$A_0(y_1)\cap A_0(y_2)\cap \ldots \cap A_0(y_s)$
by $A_0(y_1,\ldots,y_s)$ and let us call it as a {\it special subalgebra} of $A$.

Given a special subalgebra $B = A_0(y_1,\ldots,y_s)$, let us call 
$y_{s+1}\in G_{s+1}(y_1,\ldots,y_s)$ as a~special axis in $B$. 
Actually, $y_{s+1}$ is an axis in~$A$ too.

\begin{corollary}\label{series}
Let $A$ be a quasi-definite strongly axial algebra of Jordan type $1/2$ and 
$G=\{a_1,\ldots,a_n\}$ be a~generating set of $A$ consisting of axes. 
Let $e$~be a~unit of~$A$. Then there exists a~sequence 
of special subalgebras for $k\leq n$,
$$
A\supset A_0(y_1)\supset A_0(y_1,y_2)
\ldots \supset A_0(y_1,\ldots,y_k)
\supset A_0(y_1,\ldots,y_{k+1}) = (0).
$$
\end{corollary}

\begin{proof}
It follows from the proof of Lemma~\ref{lengthe}.
\end{proof}

\section*{Acknowledgments}
The work is supported by Mathematical Center in Akademgorodok 
under agreement No. 075-15-2022-281 with the Ministry of Science 
and Higher Education of the Russian Federation.

\bigskip

\noindent Ilya Gorshkov \\
\noindent Vsevolod Gubarev \\
Novosibirsk State University \\
Pirogova str. 2, 630090 Novosibirsk, Russia \\
Sobolev Institute of Mathematics \\
Acad. Koptyug ave. 4, 630090 Novosibirsk, Russia \\
e-mail: ilygor8@gmail.com, wsewolod89@gmail.com

\end{document}